\documentclass[10pt]{article}

\usepackage{amsmath}
\usepackage{amssymb}
\usepackage{amsthm}
\usepackage{bbm}
\usepackage{graphicx}
\usepackage{color}
\usepackage{epsfig}

\newtheorem{defnfoo}{Definition}[section]
\newenvironment{definition}{\begin{defnfoo}\rm}{\end{defnfoo}}
\newtheorem{notefoo}[defnfoo]{Notation}

\newtheorem{theorem}[defnfoo]{Theorem}
\newtheorem{proposition}[defnfoo]{Proposition}
\newtheorem{lemma}[defnfoo]{Lemma}
\newtheorem{corollary}[defnfoo]{Corollary}

\newtheorem{exfoo}[defnfoo]{Exercise}

\newtheorem{conjecture}[defnfoo]{Conjecture}

\newcommand{\floor}[1]{\ensuremath{\left \lfloor {#1} \right \rfloor}}
\newcommand{\ceil}[1]{\ensuremath{\left \lceil {#1} \right \rceil}}

\newcommand{\eps}{\varepsilon}
\newcommand{\N}{{\mathbb N}}
\newcommand{\Z}{{\mathbb Z}}

\newcommand{\R}{{\mathbb R}}

\newcommand{\1}{{\mathbbmss 1}}
\newcommand{\posints}{{\mathbb Z}^+}

\newcommand{\bP}{{\bf P}}

\newcommand{\bv}{{\bf v}}

\newcommand{\cO}{{\mathcal O}}

\allowdisplaybreaks

\begin{document}

\bibliographystyle{plain}

\title{Monochromatic Boxes in Colored Grids}
\author{Joshua Cooper, Stephen Fenner\thanks{Partially supported by NSF grant CCF-05-15269}, and Semmy Purewal}

\maketitle

\begin{abstract}
A $d$-dimensional {\it grid} is a set of the form $R = [a_1] \times \cdots \times [a_d]$.  A $d$-dimensional {\it box} is a set of the form $\{b_1,c_1\} \times \cdots \times \{b_d,c_d\}$.  When a grid is $c$-colored, must it admit a monochromatic box?  If so, we say that $R$ is $c$-guaranteed.  This question is a relaxation of one attack on bounding the van der Waerden numbers, and also arises as a natural hypergraph Ramsey problem (viz. the Ramsey numbers of hyperoctahedra).  We give conditions on the $a_i$ for $R$ to be $c$-guaranteed that are asymptotically tight, and analyze the set of minimally $c$-guaranteed grids.
\end{abstract}

\section{Introduction}
A $d$-dimensional {\it grid} is a set $R = [a_1] \times \cdots \times [a_d]$, where $[t] = \{1,\ldots,t\}$.  For ease of notation, we write $[a_1,\ldots,a_d]$ for $[a_1] \times \cdots \times [a_d]$.  The ``volume'' of $R$ is $\prod_{i=1}^d a_i$.  A $d$-dimensional {\it box} is a set of $2^d$ points of the form
$$
\{(x_1+\epsilon_1 s_1,\ldots,x_d+\epsilon_d s_d) \, |\, \epsilon_i \in \{0,1\} \textrm{ for } 1 \leq i \leq d\},
$$
with $s_i \neq 0$ for all $1 \leq i \leq d$. A grid $R$ is {\it $(c,t)$-guaranteed}, if for all colorings $f : R \rightarrow [c]$, there are at least $t$ distinct monochromatic boxes in $R$, i.e., boxes $B_j \subseteq R$, $j \in [t]$, so that $|f(B_j)| = 1$.  When $t=1$, we simply say that $R$ is $c$-guaranteed.  If $R$ is not $c$-guaranteed, we say it is $c$-{\it colorable}. Clearly, whether a grid is $(c,t)$-guaranteed depends only on $a_1,\ldots,a_d$.  Furthermore, if $b_i \geq a_i$ for all $i$ such that $1 \leq i \leq d$, then $[b_1,\ldots,b_d]$ is $c$-guaranteed if $[a_1,\ldots,a_d]$ is.  This ordering on $d$-tuples is sometimes called the {\it dominance order}, and we will denote it by $\preceq$.  Then one may state the above observation as the fact that the set of $c$-{\it guaranteed} grids is an up-set in the $(\N^d,\preceq)$-poset. Hence, we have a full understanding of this family if we know the minimal $c$-guaranteed grids, an antichain in the $\preceq$ order.  (Note that any such antichain is finite, a well-known fact in poset theory.)  Call the set of minimal $c$-guaranteed grids $\cO(c,d)$, the {\it obstruction set} for $c$ colors in dimension $d$.  We will focus our attention on {\it monotone} obstruction set elements, i.e., those grids for which $a_1 \leq \cdots \leq a_d$, since being $c$-guaranteed (or $(c,t)$-guaranteed) is invariant under permutations of the $a_j$.

The subject of unavoidable configurations in grids has connections with the celebrated Van der Waerden's and Szemer\'edi's Theorems.  (See, for example, \cite{AM}, \cite{FK}, and \cite{S}.)  Our results can be seen as belonging to hypergraph Ramsey theory, as follows.  Let $G$ be the complete $d$-partite $d$-uniform hypergraph with blocks of size $a_1,\ldots,a_d$.  Then an edge of $G$ can be identified with a vertex of $R = [a_1,\ldots,a_d]$ in the natural way.  Under this correspondence, a $c$-coloring of $R$ gives rise to a $c$-edge coloring of $G$, and boxes correspond precisely to subgraphs isomorphic to the ``generalized octahedron'' $K_d(2)$, the complete $d$-partite $d$-uniform hypergraph with each block of size $2$.  The generalized octahedra play an important and closely related role in the work of Kohayakawa, R\"odl, and Skokan (\cite{KRS}) on hypergraph quasirandomness.  (Among other interesting results, they show that, asymptotically, a random $c$-edge coloring of $G$ has the fewest number of monochromatic $K_d(2)$'s possible.)  We may translate each of our results into statements about the Ramsey numbers of hyperoctahedra-free $d$-partite $d$-uniform graphs.  For example, in Section \ref{sec:threeD}, we give a family of upper bounds on the sizes of $3$-dimensional grids which have a $2$-coloring admitting no monochromatic box; this is equivalent to asking for the extremal tripartite $3$-uniform hypergraphs which are $(K_3(2),K_3(2))$-Ramsey.

The present work is even more closely connected to the ``Product Ramsey Theorem.''  Though the proof appears in \cite{GRS}, the statement appearing in \cite{T} best illustrates the connection:

\begin{theorem}[Product Ramsey Theorem] Let $k_1,\ldots,k_d$ be nonnegative integers; let $c$ and $d$ be positive integers; and let $m_1,\ldots,m_d$ be integers with $m_i \geq k_i$ for $i \in [d]$.  Then there exists an integer $R = R(c,d;k_1,\ldots,k_d;m_1,\ldots,m_d)$ so that if $X_1,\ldots,X_d$ are sets and $|X_i| \geq R$ for $i \in [d]$, then for every function $f : \binom{X_1}{k_1} \times \cdots \times \binom{X_d}{k_d} \rightarrow [c]$, there exists an element $\alpha \in [c]$ and subsets $Y_1,\ldots,Y_d$ of $X_1,\ldots,X_d$, respectively, so that $|Y_i| \geq m_i$ for $i \in [d]$ and $f$ maps every element of $\binom{X_1}{k_1} \times \cdots \times \binom{X_d}{k_d}$ to $\alpha$.
\end{theorem}

This result ensures that the quantity $N(c,d) = R(c,d;1^d;2^d)$, which corresponds to the least $R$ so that $[R]^d$ is $c$-guaranteed, is finite.  A closer analysis of $N(c,d)$ -- in fact, the more general $N(c,d,m) = R(c,d;1^d;m^d)$ -- appears in the manuscript \cite{ADS} by Agnarsson, Doerr, and Schoen.  They obtain asymptotic bounds on $N(c,d,m)$ that are valid for large $m$.  Here, we examine instead the least nontrivial case of $m=2$, and consider grids which are not necessarily equilateral.

In the next section, we show that any grid of sufficiently small volume (approximately $c^{2^d-1}$) is $c$-colorable.  The following section shows that the analysis is tight: there are grids of this volume which are $c$-guaranteed.  Not all grids of sufficient volume are $c$-guar\-an\-teed, although Section \ref{sec:hereditary} demonstrates that any grid all of whose lower-dimensional subgrids are sufficiently voluminous is indeed $c$-guar\-an\-teed.  The next section gives a tight upper bound on the volume of minimally $c$-guaranteed grids, i.e., elements of the obstruction set.  Section \ref{sec:Ocdsize} then addresses the question of how many obstructions there are.  Finally, as mentioned above, Section \ref{sec:threeD} considers the case of $c=2$ and $d=3$, where some interesting computational questions arise.  This extends work of the second two authors (\cite{FGGP}) for $d=2$ and $2 \leq c \leq 4$.

Throughout the present manuscript, unless we explicitly say otherwise, we use the notations $x = O(y)$ and $y = \Omega(x)$ to mean that there is a function $F : \posints \rightarrow \posints$ such that $x \le F(d) y$.  That is, $x$ is bounded by $y$ times a number that only depends on $d$.  (Naturally, $x = \Theta(y)$ means that $x = O(y)$ and $x = \Omega(y)$, and notation $x=o(y)$ is defined analogously.)  In general, $x$ and $y$ will depend on $c$, $d$, and perhaps other quantities.

\section{All small grids are $c$-colorable}

Define $V(c,d)$ to be the largest integer $V$ so that every $d$-dimensional grid $R$ with volume at most $V$ is $c$-colorable.  Below, we show that $V(c,d)$ is $\Theta(c^{2^d-1})$.

\begin{theorem} $$V(c,d) = \Omega(c^{{2}^d-1}).$$
%where $o(1)$ tends to $0$ as $\max\{c,d\} \rightarrow \infty$.
In fact, $V(c,d) > c^{2^d-1}/e2^d$, where $e = 2.718\ldots$ is Napier's constant.
\end{theorem}
\begin{proof} We apply the Lov\'asz Local Lemma (see, e.g., \cite{AS}), which states the following.  Suppose that $A_1,\ldots,A_t$ are events in some probability space, each of probability at most $p$.  Let $G$ be a ``dependency'' graph with vertex set $\{A_i\}_{i=1}^t$, i.e., a graph so that, whenever a set $S$ of vertices induces no edges in $G$, then $S$ is a mutually independent family of events.  Then $\bP(\bigwedge_{i=1}^t \overline{A_i}) > 0$ if $ep(\Delta + 1) \leq 1$, where $\Delta = \Delta(G)$ is the maximum degree of $G$.

Now, suppose $R=[a_1,\ldots,a_d]$ is a grid of volume $V$, and we color the points of $R$ uniformly at random from $[c]$.  Enumerate all boxes in $R$ as $B_1,\ldots,B_t$.  Define $A_i$ to be the event that $B_t$ is monochromatic in this random coloring. Clearly, we may take $G$ to have an edge between $A_i$ and $A_j$ whenever $B_i \cap B_j \neq \emptyset$.  The degree of a vertex $A_i$ is then the number of boxes $B_j$, $j \neq i$, which intersect $B_i$.  Since we may specify the list of all such boxes by choosing one of the $2^d$ points of $B_i$, and then choosing the $d$ coordinates of its antipodal point, $\deg_G(A_i)$ is at most
$$
2^d \prod_{i=1}^d (a_i - 1) - 1 < 2^d \prod_{i=1}^d a_i - 1 = 2^d V - 1.
$$
(The outermost $-1$ here reflects the fact that $B_i$ may be excluded among these choices.)  The probability of each $A_i$ is the same: $p = c^{-2^d+1}$.  Therefore,
$$
ep(\Delta+1) < e c^{-2^d + 1} 2^d V
$$
which is $\leq 1$ whenever $V \leq c^{2^d - 1}/e 2^d$.
\end{proof}

\section{Some large grids are $c$-guaranteed}

\begin{theorem} \label{thm:CSrawupperbound} Fix $c, d$, define $R = [a_1,\ldots,a_d]$, and let $M = \prod_i \binom{a_i}{2}$ denote the total number of boxes in $R$.  For $\min\{a_1,\ldots,a_d\} \rightarrow \infty$, $R$ is $(c,M(1+o(1))/c^{2^d-1})$-guaranteed.
\end{theorem}

Theorem~\ref{thm:CSrawupperbound} follows quickly from the next lemma, whose extra strength we will need later.

\begin{lemma}\label{lem:box-bound}
Suppose $c\geq 1$.  For $d\geq 1$ and integers $a_1,\ldots,a_d \geq 2$, let $M = \prod_{i=1}^d \binom{a_i}{2}$.  The grid $R=[a_1,\ldots,a_d]$ is $(c,M\Delta_d/c^{2^d-1})$-guaranteed provided $\Delta_1,\ldots,\Delta_d > 0$, where $\Delta_j$, $0 \leq j \leq d$, is given by the recurrence
\begin{eqnarray*}
\Delta_0 & = & 1, \\
\Delta_{j} & = & \Delta_{j-1}^2\left(1 - \frac{\frac{c^{2^{j-1}}}{\Delta_{j-1}} - 1}{a_{j} - 1}\right).
\end{eqnarray*}
\end{lemma}

%\begin{proof}
%This is precisely what one obtains by chasing the expression $1+o(1)$ through the argument of Theorem \ref{thm:CSrawupperbound}.
%\end{proof}

\begin{proof}
We proceed inductively.  Suppose $d=1$, let $f : [a_1] \rightarrow [c]$ be a $c$-coloring, and define
$$
\gamma_i = |f^{-1}(i)|
$$
to be the number of points colored $i$, $1 \leq i \leq c$.  Then the number $N$ of monochromatic boxes in $f$ is exactly
$$
N=\sum_{i=1}^c \binom{\gamma_i}{2} = \frac{1}{2} \cdot \sum_{i=1}^c \left( \gamma_i^2 - \gamma_i \right ) = \frac{1}{2} \cdot \left ( \sum_{i=1}^c \gamma_i^2 - a_1 \right ).
$$
Applying Cauchy-Schwarz,
\begin{align*}
N & \geq  \frac{\left ( \sum_{i=1}^c \gamma_i \right )^2}{2c} - \frac{a_1}{2} = \frac{a_1^2}{2c} - \frac{a_1}{2} \\
& = \frac{a_1 (a_1 - c)}{2c} = \frac{1}{c}\binom{a_1}{2}\frac{a_1-c}{a_1-1} = \frac{1}{c}\binom{a_1}{2}\Delta_1.
\end{align*}
Now, suppose the statement is true for dimensions $< d+1$, and consider a coloring $f : [a_1,\ldots,a_{d+1}] \rightarrow [c]$.  Consider the $a_{d+1}$ colorings $f_j$ of the $d$-dimensional grid $[a_1,\ldots,a_{d}]$ induced by setting the last coordinate to $j$, i.e.,
$$
f_j (x_1,\ldots,x_{d}) = f(x_1,\ldots,x_{d},j).
$$
Let $\gamma_i(B)$, for a box $B \subset [a_1,\ldots,a_{d}]$ and $i \in [c]$, denote the number of $j$ so that $f_j|_B \equiv i$.  Then the number $N$ of monochromatic $(d+1)$-dimensional boxes in $f$ is
\begin{align*}
N & = \sum_i \sum_{B} \binom{\gamma_i(B)}{2} \\
& = \frac{1}{2} \cdot \sum_i \sum_{B} (\gamma_i(B)^2 - \gamma_i(B)) \\
& \geq \frac{\left ( \sum_{B} \sum_i \gamma_i(B) \right )^2}{2Mc} - \frac{1}{2} \cdot \sum_{B} \sum_i \gamma_i(B) \\
& = \frac{\left ( \sum_{B} \sum_i \gamma_i(B) \right )^2 - Mc \sum_{B} \sum_i \gamma_i(B)}{2Mc}
\end{align*}
where $M = \prod_{i=1}^{d} \binom{a_i}{2}$. Since, by the inductive hypothesis, $f_j$ induces at least $M \Delta_d/c^{2^{d}-1}$ monochromatic boxes,
$$
\sum_i \sum_B \gamma_i(B) \geq \frac{a_{d+1} M\Delta_d}{c^{2^{d}-1}},
$$
so that
\begin{align*}
N & \geq \frac{a_{d+1}^2 M^2\Delta_d^2/c^{2^{d+1}-2} - a_{d+1} M^2 \Delta_d c/c^{2^{d}-1}}{2Mc}\\
& = \frac{a_{d+1} M (a_{d+1}\Delta_d^2 - c^{2^{d}}\Delta_d)}{2 c^{2^{d+1}-1}} \\
& = \frac{M}{c^{2^{d+1}-1}} \binom{a_{d+1}}{2}\frac{a_{d+1}\Delta_d^2 - c^{2^d}\Delta_d}{a_{d+1} - 1} \\
& = \frac{\prod_{i=1}^{d+1}\binom{a_i}{2}}{c^{2^{d+1}-1}}\cdot \Delta_d^2\left(\frac{a_{d+1} - c^{2^d}/\Delta_d}{a_{d+1}-1}\right) \\
& = \frac{\prod_{i=1}^{d+1}\binom{a_i}{2}\Delta_{d+1}}{c^{2^{d+1}-1}}.
\end{align*}
\end{proof}

\begin{proof}[Proof of Theorem~\ref{thm:CSrawupperbound}]
Fix $c,d\ge 1$.  It is clear by induction on $j$ that for
all $1\le j\le d$, as $\min\{a_1,\ldots,a_d\} \rightarrow \infty$, \
$\Delta_j = 1+o(1)$, and so in particular, $\Delta_j > 0$ if
$\min\{a_1,\ldots,a_d\}$ is large enough.
\end{proof}

Note that, in the notation of Lemma~\ref{lem:box-bound}, if $\Delta_1,\ldots,\Delta_d > 0$, then $[a_1,\ldots,a_d]$ is not $c$-colorable.  Therefore we may conclude the following.

\begin{corollary}
In the notation of Lemma~\ref{lem:box-bound}, let $\Gamma_j$, $0 \leq j \leq d$, be given by the recurrence
\begin{eqnarray*}
\Gamma_0 & = & 1, \\
\Gamma_{j} & = & \Gamma_{j-1}^2\left(1 - \frac{c^{2^{j-1}}/\Gamma_{j-1}}{a_{j} - 1}\right) = \Gamma_{j-1} \left(\Gamma_{j-1} - \frac{c^{2^{j-1}}}{a_{j} - 1}\right).
\end{eqnarray*}
If $\Gamma_1,\ldots,\Gamma_d > 0$, then $[a_1,\ldots,a_d]$ is $c$-guaranteed.
\end{corollary}

\begin{proof}
Assume $\Gamma_1,\ldots,\Gamma_d > 0$.  A routine induction shows that $\Gamma_j \leq \Delta_j$ for $0\leq j\leq d$.
\end{proof}

\begin{lemma}\label{lem:eps-bound}
In the notation of Lemma~\ref{lem:box-bound}, let $\eps_j$ be given by the recurrence
\begin{eqnarray*}
\eps_0 & = & 0, \\
\eps_{j} & = & 2\eps_{j-1} + \frac{c^{2^{j-1}}}{a_{j} - 1}.
\end{eqnarray*}
If $\eps_d < 1$, then $[a_1,\ldots,a_d]$ is $c$-guaranteed.
\end{lemma}

\begin{proof}
Clearly, $0=\eps_0 < \eps_1 < \eps_2 < \cdots < \eps_d$, and so by
assumption $\eps_i < 1$ for all $i\in[d]$.  An induction on $i$
shows that $\Gamma_i \ge 1 - \eps_i$ for $0\le i\le d$: This is
clearly true for $i = 0$.  Suppose $i<d$ and $\Gamma_i \ge 1 -
\eps_i$.  Then setting $\eta := c^{2^i}/(a_{i+1}-1)$ and noting that
$\Gamma_i \ge 0$, we have
\begin{equation}\label{eqn:Gamma-lb}
\Gamma_{i+1} = \Gamma_i(\Gamma_i - \eta) \ge \Gamma_i(1 - \eps_i - \eta).
\end{equation}
The term in the parentheses is positive:
$$
1 - \eps_i - \eta \ge 1 - 2\eps_i - \eta = 1 - \eps_{i+1} > 0
$$
by assumption.  Thus continuing (\ref{eqn:Gamma-lb}) and using the inductive hypothesis again,
$$
\Gamma_i(1 - \eps_i - \eta) \ge (1 - \eps_i)(1 - \eps_i - \eta) \ge 1 - 2\eps_i - \eta  = 1 - \eps_{i+1}.
$$
\end{proof}

Motivated by the preceding lemma, for every $c,d\ge 1$ and grid $R = [a_1,\ldots,a_d]$ with $a_i \ge 2$ for all $i\in [d]$, we define
$$
\varepsilon_c(R) := \sum_{i=1}^d 2^{d-i}\frac{c^{2^{i-1}}}{a_i - 1}.
$$

\begin{lemma}\label{lem:key-fact}
If $R=[a_1,\ldots,a_d]$ is not $c$-guaranteed, then $\varepsilon_c(R) \ge 1$.
\end{lemma}

\begin{proof}
We let $\varepsilon_j := \varepsilon_c([a_1,\ldots,a_j]) =
\sum_{i=1}^j 2^{j-i}c^{2^{i-1}}/(a_i-1)$ for all $j$
with $0\le j\le d$, and notice that the $\varepsilon_j$ satisfy
the recurrence in Lemma~\ref{lem:eps-bound}.
\end{proof}

\begin{corollary}\label{cor:box-bound}
For any fixed $d\ge 1$ and $c\ge 2$, if $n$ is least such that $[n]^d$ is $c$-guaranteed, then $n < (d+2) c^{2^{d-1}}$.  Furthermore,
$$
2^{-d} e^{-1} < \frac{V(c,d)}{c^{2^d-1}} < (d+2)^d 2^{d(d-1)/2}.
$$
\end{corollary}

\begin{proof}
%For $0\leq i \leq d$,
%$$
%\eps_c(R_i) = \sum_{j=1}^i 2^{i-j}\frac{c^{2^{j-1}}}{a_j-1}.
%$$
%Hence,
If we take $a_j = (d+1)2^{d-j}c^{2^{j-1}}+1$ for all $1 \leq j \leq
d$, then $\eps_c(R) = d/(d+1) < 1$.  The second result now follows from the fact that
\begin{align*}
\prod_{j=1}^d a_j &< \prod_{j=1}^d (d+2)2^{d-j}c^{2^{j-1}} \\
&= (d+2)^d 2^{\sum_{j=1}^d (d-j)}c^{\sum_{j=1}^d 2^{j-1}} \\
&= (d+2)^d 2^{\sum_{j=1}^{d-1} j}c^{\sum_{j=0}^{d-1} 2^{j}} \\
&= (d+2)^d 2^{d(d-1)/2}c^{2^d - 1}.
\end{align*}
The first result follows by taking $n:=a_d$.
\end{proof}

\section{Hereditarily large grids are $c$-guaranteed} \label{sec:hereditary}

It is possible for grids of arbitrarily large volume to be $c$-colorable.  Indeed, one need only have one of the dimensions be at most $c$, and then color the grid with this coordinate.  However, if we require that each lower dimensional sub-grid be sufficiently voluminous, then the whole grid is $c$-colorable.  This statement is made precise by the following theorem.

\begin{theorem}\label{thm:subgrids}
Fix $d>0$, and define $C_j = (d 2^d)^{\frac{3}{2} (3^{j-1}-1)}$ for $j \geq 1$.  For all integers $c \ge 1$ and $1\le a_1\le a_2\le \cdots \le a_d$, if $\prod_{i=1}^j a_i > C_j c^{(3^j - 1)/2}$ for all $j\in{[d]}$, then $[a_1,\ldots,a_d]$ is $c$-guaranteed.
\end{theorem}

We require a lemma and a bit of notation: If $R = [a_1,\ldots,a_d]$ and $1\leq j < d$, let $R_j$ denote $[a_1,\ldots,a_j]$ and let $\overline{R}_j$ denote $[a_{j+1},\ldots,a_d]$.  Note that, if $R$ is $c$-guaranteed, then $R_j$ is as well.  Indeed, if $f : R_j \rightarrow [c]$ is a $c$-coloring of $R_j$, then the function $g : R \rightarrow [c]$ defined by $g(x_1,\ldots,x_d) = f(x_1,\ldots,x_j)$ is a $c$-coloring of $R$.  We will also make repeated use of the following easily verified fact: For every integer $j\ge 0$, $j\cdot 2^{j-1} \le (3^j - 1)/2$ and $j\cdot 2^j + 1 \le 3^j$.

\begin{lemma}\label{lem:virtual-colors}
Let $c\ge 1$, let $R = [a_1,\ldots,a_d]$ be a grid, and let $j\in {[d-1]}$.  Define
$$
c^\prime := c\cdot \prod_{i=1}^j\binom{a_i}{2} \le 2^{-j} \cdot c \cdot \prod_{i=1}^j a_i^2.
$$
If $R_j$ is $c$-guaranteed and $\overline{R}_j$ is $c^\prime$-guaranteed, then $R$ is $c$-guaranteed.
\end{lemma}

\begin{proof}
Assume that $R_j$ is $c$-guaranteed and that $\overline{R}_j$ is $c^\prime$-guaranteed.  Suppose that $f : R \rightarrow [c]$ is a $c$-coloring.  Consider the coloring $g : \overline{R}_j \rightarrow [c^\prime]$ that assigns the pair $(B,s)$ to the point $\bv$, $B$ being an arbitrary choice of $j$-dimensional box colored monochromatically by $f_j : R_j \rightarrow [c]$, where $f_j(x_1,\ldots,x_j) = f(x_1,\ldots,x_j,\bv)$, and $s$ being its color.  (Note that $R_j$ is $c$-guaranteed, so such a $B$ always exists.) Then $g$ is a $c^\prime$-coloring, because there are exactly $c^\prime$ many different $(B,s)$.
Since $\overline{R}_j$ is $c^\prime$-guaranteed, $g$ colors some $(d-j)$-dimensional box $B_1$ monochromatically, with color $(B_2,s)$.  But then $B_2 \times B_1$ is a $d$-dimensional box monocolored by $f$ with color $s$.
\end{proof}

\begin{proof}[Proof of Theorem \ref{thm:subgrids}.]  The statement is clearly true when $d=1$ since $C_1 = 1$.  Suppose $d>1$ and the statement is true for all $d^\prime < d$.  Let $R = [a_1, \cdots, a_d]$ be a monotone grid satisfying the hypothesis of the theorem.

\emph{Case 1:} $\varepsilon_c(R) < 1$.  The result follows immediately from Lemma~\ref{lem:key-fact}.

\emph{Case 2:} $\varepsilon_c(R) \geq 1$.  Then there is some $j\in {[d]}$ such that $2^{d-j}c^{2^{j-1}}/(a_j-1) \geq 1/d$, i.e.,
$$
a_j \leq d 2^{d-j} c^{2^{j-1}} +1 < d 2^{d-j+1} c^{2^{j-1}}.
$$
Since $j2^j \leq 3^j-1$ for all integers $j \geq 1$,
$$
\prod_{i=1}^j a_i \le \prod_{i=1}^j a_j < d^j 2^{j(d-j+1)} c^{j2^{j-1}} \leq d^j 2^{j(d-j+1)} c^{(3^j-1)/2},
$$
and so for all $k\in {[d-j]}$,
$$
\prod_{i=1}^k a_{j+i} > \frac{C_{j+k}}{d^j 2^{j(d-j+1)}} c^{(3^{j+k} - 1)/2 - (3^j-1)/2} \geq \frac{C_{j+k}}{d^j 2^{j(d-j+1)}} c^{3^j(3^k - 1)/2}.
$$
Let $c^\prime = d^{2j}2^{2j(d-j+1)}c^{3^j}$.  (Note that $c'\ge
c\cdot\prod_{i=1}^j a_i^2$.)  Then for all $k\in {[d-j]}$,
\begin{align*}
\prod_{i=1}^k a_{j+i} &> \frac{C_{j+k}}{d^j 2^{j(d-j+1)}} \left (\frac{c^\prime}{d^{2j}2^{2j(d-j+1)}} \right)^{(3^k - 1)/2} \\
&= \frac{(d 2^d)^{\frac{3}{2} (3^{j+k-1}-1)}}{(d 2^{d-j+1})^{j3^k}} {c^\prime}^{(3^k - 1)/2} \\
&\ge (d 2^d)^{\frac{3}{2} (3^{j+k-1}-1) - j3^k} {c^\prime}^{(3^k - 1)/2} \\
&\ge (d 2^d)^{\frac{3}{2} (3^{j+k-1}-1) - (3^{j} - 1) 3^k/2} {c^\prime}^{(3^k - 1)/2},
\end{align*}
because $j \leq (3^j-1)/2$ for all $j \geq 1$.  Continuing the computation,
\begin{align*}
\prod_{i=1}^k a_{j+i} &> (d 2^d)^{\frac{3}{2} (3^{j+k-1}-1) - (3^{j} - 1) 3^k/2} {c^\prime}^{(3^k - 1)/2} \\
&= (d 2^d)^{\frac{3}{2} (3^{j+k-1}-1 - 3^{j+k-1} + 3^{k-1})} {c^\prime}^{(3^k - 1)/2} \\
&= (d 2^d)^{\frac{3}{2} (3^{k-1}-1)} {c^\prime}^{(3^k - 1)/2} \\
&= C_k {c^\prime}^{(3^k - 1)/2}.
\end{align*}
Therefore $\overline{R}_j = [a_{j+1},\ldots,a_d]$ is $c^\prime$-guaranteed by the inductive hypothesis.  (It is easy to see that the $C_j$'s are increasing in $d$, so taking $d^\prime = d-j$ causes no problem here.)  Since $R_j$ is also $c$-guaranteed by the inductive hypothesis, we may apply Lemma \ref{lem:virtual-colors} to conclude that $R$ is $c$-guaranteed.
\end{proof}

\section{Upper bounds on the volume of obstruction grids}

Before proceeding, we introduce the following notation.  For $d\geq 1$ and any monotone grid $R = [a_1,\ldots,a_d]$ where $a_d >1$, we let $R^-$ denote the monotone grid obtained from $R$ by subtracting one from $a_j$, where $j\in [d]$ is least such that $a_j = a_d$.  Note that if $R$ is monotone and $R \in \cO(c,d)$, then $R$ is $c$-guaranteed but $R^-$ is not $c$-guaranteed.

The next theorem gives an asymptotic upper bound on the volume $\prod_{i=1}^d a_i$ of any grid $[a_1,\ldots,a_d] \in \cO(c,d)$.

\begin{theorem}\label{thm:volume}
For every $d\ge 1$ and every grid $R = [a_1,\ldots,a_d] \in \cO(c,d)$,
$$
\prod_{i=1}^d a_i = O\left( c^{(3^d-1)/2} \right).
$$
\end{theorem}

The theorem follows immediately from the following lemma:

\begin{lemma}\label{lem:pinch-points}
For every $d\ge 1$, every $c\ge 2$, and every monotone grid $R = [a_1,\ldots,a_d] \in \cO(c,d)$, there is a set $P\subseteq [d]$ such that
\begin{enumerate}
\item
$d\in P$,
\item\label{item:vol-bound}
$\prod_{i=1}^\ell a_i = O\left( c^{(3^\ell-1)/2} \right)$ for every $\ell\in P$, and
\item\label{item:ak-bound}
For every $k\in [d]$,
$$
a_k = O\left( c^{3^j\cdot 2^{\ell - j - 1}} \right),
$$
where $\ell$ is the least element of $P$ that is $\ge k$, and $j$ is the biggest element of $P$ that is $<k$ ($j=0$ if there is no such element).
\end{enumerate}
(We call the elements of $P$ \emph{pinch points} for $R$.)
\end{lemma}

\begin{proof}
Let $d\ge 1$ and $c\ge 2$ be given, and let $R = [a_1,\ldots,a_d] \in \cO(c,d)$ be a monotone grid.  Then $R$ is $c$-guaranteed, and thus $R_j$ is also $c$-guaranteed for all $1\le j\le d$.  Since $R\in\cO(c,d)$, we have that $R^-$ is not $c$-guaranteed, and thus $\varepsilon_c(R^-) \ge 1$.  This in turn implies that there is some largest $\ell \in [d]$ such that
$$
2^{d-\ell}\frac{c^{2^{\ell-1}}}{a_\ell - 2} \ge \frac{1}{d}.
$$
(Note that the denominator is positive, because $a_\ell \ge a_1 \ge c+1 \ge 3$ since $R$ is $c$-guaranteed.)  Thus,
\begin{equation}\label{eqn:aj-bound}
a_\ell \le d 2^{d-\ell}\cdot c^{2^{\ell-1}} + 2 \le (d+2)2^{d-\ell} \cdot c^{2^{\ell-1}},
\end{equation}
and thus
\begin{equation}\label{eqn:vol-bound}
\prod_{i=1}^\ell a_i \le (a_\ell)^\ell \le \left((d+2)2^{d-\ell}\right)^\ell \cdot c^{\ell\cdot 2^{\ell-1}} \le \left((d+2)2^{d-1}\right)^d \cdot c^{(3^\ell-1)/2},
\end{equation}
which implies that $\ell$ satisfies Condition~\ref{item:vol-bound} of the lemma.  We will make $\ell$ the least element of $P$, noticing that Equation~(\ref{eqn:aj-bound}) and the monotonicity of $R$ imply that $a_k$ satisfies Condition~\ref{item:ak-bound} of the lemma for all $k\in[\ell]$ (with $j = 0$).

If $\ell = d$, then we let $P = \{\ell\} = \{d\}$ and we are done.

Otherwise, $\ell < d$.  Note that $R^- = R_\ell\times (\overline{R}_\ell)^-$ up to a possible permutation of the coordinates.  Recall also that $R_\ell$ is $c$-guaranteed, but $R^-$ is not.  It follows from Lemma~\ref{lem:virtual-colors} that $(\overline{R}_\ell)^-$ is not $c^\prime$-guaranteed, where
$$
c^\prime := c\cdot \prod_{i=1}^\ell \binom{a_i}{2} = O\left(c\cdot\left(\prod_{i=1}^\ell a_i\right)^2\right).
$$
The bound in Equation~(\ref{eqn:vol-bound}) gives $c^\prime = O\left(c^{3^\ell}\right)$.

We thus have $\varepsilon_{c^\prime}((\overline{R}_\ell)^-) \ge 1$, and so there is some largest $m$ with $\ell < m \le d$ such that
$$
2^{d-m}\frac{(c^\prime)^{2^{m-\ell-1}}}{a_m - 2} \ge \frac{1}{d-\ell},
$$
which gives
\begin{eqnarray}\label{eqn:am-bound-first}
a_m & \le & (d-\ell) 2^{d-m}\cdot (c^\prime)^{2^{m-\ell-1}} + 2 \\
& \le & (d-\ell+2)2^{d-m} \cdot (c^\prime)^{2^{m-\ell-1}} \\
& = & O\left(c^{3^\ell\cdot 2^{m-\ell-1}}\right).
\label{eqn:am-bound-last}
\end{eqnarray}
For the volume of $R_m$, we get
\begin{eqnarray*}
\prod_{i=1}^m a_i & = & \prod_{i=1}^\ell a_i \cdot \prod_{i=\ell+1}^m a_i \\
& \le & \left(\prod_{i=1}^\ell a_i\right) \cdot (a_m)^{m-\ell} \\
& = & O\left(c^{(3^\ell-1)/2}\right)\cdot O\left(c^{3^\ell\cdot (m-\ell)\cdot 2^{m-\ell-1}}\right) \\
& = & O\left(c^{(3^\ell-1)/2}\cdot c^{3^\ell\cdot (3^{m-\ell}-1)/2}\right) \\
& = & O\left(c^{(3^m-1)/2}\right).
\end{eqnarray*}
We make $\ell$ and $m$ the two least elements of $P$, and the last calculation shows that $m\in P$ satisfies Condition~\ref{item:vol-bound}.  Further, since $a_k \le a_m$ for all $k$ such that $\ell < k \le m$, Condition~\ref{item:ak-bound} is also satified for all these $a_k$ by Equations~(\ref{eqn:am-bound-first})--(\ref{eqn:am-bound-last}).

If $m = d$, then we let $P = \{\ell,m\}$ and we are done.  Otherwise, we repeat the argument above using $m$ instead of $\ell$ to obtain an $n$ with $m < n \le d$ such that $\ell$, $m$, and $n$ being the least three elements of $P$ satisfies Conditions~\ref{item:vol-bound} and \ref{item:ak-bound} of the lemma, and so on until we arrive at $d$, whence we set $P := \{\ell,m,n,\ldots,d\}$.
\end{proof}

The next proposition shows that the bounds in Lemma~\ref{lem:pinch-points} are asymptotically tight.

\begin{proposition}
For $c\ge 2$, there is an infinite sequence $\{\mu_j(c)\}_{j=1}^\infty$ of positive integers such that
\begin{enumerate}
\item\label{item:lb}
$\mu_j(c) \ge 1 + 2^{(1-3^{j-1})/2}\cdot c^{3^{j-1}}$ for all $j\in \posints$, and
\item\label{item:os}
for all $d\ge 1$, the grid $[\mu_1(c),\ldots,\mu_d(c)] \in \cO(c,d)$  with pinch point set $P = [d]$.
\end{enumerate}
\end{proposition}

\begin{proof}
For all $c \ge 2$, define
\begin{eqnarray*}
\mu_1(c) & := & 1 + c, \\
\mu_2(c) & := & 1 + c\cdot\binom{c+1}{2}, \\
& \vdots & \\
\mu_{j+1}(c) & := & 1 + c\cdot \prod_{i=1}^j \binom{\mu_i(c)}{2}, \\
& \vdots &
\end{eqnarray*}
Fix $c\ge 2$ and let $\mu_j$ denote $\mu_j(c)$ for short.  A routine induction on $j$ shows (\ref{item:lb}).  For the inductive step, noting that $\sum_{i=0}^{j-1} 3^i = (3^j - 1)/2$, we have
\begin{align*}
\mu_{j+1} &= 1 + c\cdot \prod_{i=1}^j \binom{\mu_i}{2} \\
&\ge 1 + \frac{c}{2^j} \prod_{i=1}^j (\mu_i - 1)^2 \\
&\ge 1 + \frac{c}{2^j}\prod_{i=1}^j \frac{c^{2\cdot 3^{i-1}}}{2^{3^{i-1}-1}} \\
& = 1 + \frac{c^{3^j}}{2^{(3^j-1)/2}}.
\end{align*}
For (\ref{item:os}), we use induction on $d\ge 1$ to show separately that
\begin{enumerate}
\item
$[\mu_1,\ldots,\mu_d]$ is $c$-guaranteed, and
\item\label{item:not-c-two-guar}
$[\mu_1,\ldots,\mu_d]$ is not $(c,2)$-guaranteed (i.e., there is a coloring $[\mu_1,\ldots,\mu_d]\rightarrow [c]$ that monocolors exactly one box).
\end{enumerate}
Clearly $[\mu_1] = [1+c]$ is $c$-guaranteed by the Pigeonhole Principle.  Now let $d \ge 2$ and assume that $[\mu_1,\ldots,\mu_{d-1}]$ is $c$-guaranteed.  Then letting $c^\prime = c\cdot\prod_{i=1}^{d-1} \binom{\mu_i}{2}$, we have $\mu_d = 1+c^\prime$, and hence $[\mu_d]$ is $c^\prime$-guaranteed.  But then, $[\mu_1,\ldots,\mu_d]$ is $c$-guaranteed by Lemma~\ref{lem:virtual-colors} (letting $ j = d-1$).

Now for claim (\ref{item:not-c-two-guar}).  For $d=1$, clearly the coloring $[\mu_1]\rightarrow [c]$ mapping $j\mapsto (j\bmod c) + 1$ has exactly one monochromatic $1$-dimensional box, namely, $(1;c) = \{1,c+1\}$.  Now let $d\ge 2$ and assume claim (\ref{item:not-c-two-guar}) holds for $d-1$, i.e., there is a coloring $[\mu_1,\ldots,\mu_{d-1}]\rightarrow [c]$ that monocolors exactly one box.  We will call such a coloring \emph{minimal}.  This generates exactly $\prod_{i=1}^{d-1} \binom{\mu_i}{2}$ many boxes in $[\mu_1,\ldots,\mu_{d-1}]$.  For each of these boxes $B$ and for each color $s$, we can find a minimal coloring that monocolors $B$ with $s$ by permuting the order of the hyperplanes along each axis and by permuting the colors.  Thus there are exactly $c^\prime = c\cdot\prod_{i=1}^{d-1} \binom{\mu_i}{2}$ many distinct minimal colorings.  We overlay these $c^\prime$ many colorings to obtain a coloring of $[\mu_1,\ldots,\mu_{d-1},c^\prime]$ with no monochromatic $d$-boxes.  We then duplicate the first $(d-1)$-dimensional layer to arrive at a $c$-coloring of $[\mu_1,\ldots,\mu_{d-1},1+c^\prime] = [\mu_1,\ldots,\mu_d]$.  This coloring has only one monocolored $d$-box: the box corresponding to the duplicated layer of unique monocolored $(d-1)$-boxes.  This shows Item~(\ref{item:not-c-two-guar}).

It follows from claim (\ref{item:not-c-two-guar}) that $[\mu_1,\ldots,\mu_{d-1},\mu_d-1]$ is not $c$-guaranteed for any $d\ge 1$, since we can remove a single hyperplane from the only monocolored $d$-box in some minimal coloring of $[\mu_1,\ldots,\mu_{d-1},\mu_d]$ to leave a coloring of $[\mu_1,\ldots,\mu_{d-1},\mu_d-1]$ without any monochromatic $(d-1)$-boxes.  From this it easily follows that $[\mu_1,\ldots,\mu_d]\in\cO(c,d)$, because $[\mu_1,\ldots,\mu_{j-1},\mu_j-1]$ is not $c$-guaranteed, and hence $[\mu_1\ldots,\mu_{j-1},\mu_j-1,\mu_{j+1},\ldots,\mu_d]$ is not $c$-guaranteed, for any $j\in [d]$.

Finally, it is evident that all $j\in[d]$ are pinch points for $[\mu_1,\ldots,\mu_d]$.  (It is interesting to note that $[\mu_1,\ldots,\mu_d]$ is the lexicographically first element of $\cO(c,d)$.)
\end{proof}

\section{Upper bound on the size of the obstruction set} \label{sec:Ocdsize}

It was shown in \cite{FGGP} that $|\cO(c,2)| \le 2c^2$.  We give an asymptotic upper bound for $|\cO(c,d)|$ for every fixed $d\ge 3$.

\begin{theorem}
For all $d \ge 3$,
$$
|\cO(c,d)| = O\left(c^{(17\cdot 3^{d-3} - 1)/2}\right).
$$
\end{theorem}

\begin{proof}
Fix $d \ge 3$.  We give an asymptotic upper bound on the number of monotone grids in $\cO(c,d)$.  The size of $\cO(c,d)$ is at most $d!$ times this bound, and so it is asymptotically equivalent.  By Lemma~\ref{lem:pinch-points}, every grid $R\in\cO(c,d)$ has a set $P$ of pinch points.  For each set $P\subseteq [d]$ such that $d\in P$, let $\#_c(P)$ be the number of monotone grids in $\cO(c,d)$ having pinch point set $P$.  There are $2^{d-1}$ many such $P$, so an asymptotic bound on $\max\{\#_c(P) \mid P\subseteq[d]\mathrel{\wedge} d\in P \}$ gives the same asymptotic bound on $|\cO(c,d)|$.

Fix a set $P\subseteq[d]$ such that $d\in P$, and let $P = \{\ell_1 < \ell_2 < \cdots < \ell_s = d\}$, where $s = |P|$ and $\ell_1,\ldots,\ell_s$ are the elements of $P$ in increasing order.  For convenience, set $\ell_0 := 0$.  Lemma~\ref{lem:pinch-points} says that for any monotone grid $R = [a_1,\ldots,a_d]\in\cO(c,d)$ having pinch point set $P$, for any $b\in [s]$, and for any $k$ such that $\ell_{b-1} < k \le \ell_b$, we have $a_k = O\left(c^{e(b)}\right)$, where
$$
e(b) := 3^{\ell_{b-1}}\cdot 2^{\ell_b - \ell_{b-1} - 1}.
$$
To bound $\#_c(P)$, we first note that for any choice of $1\le a_1\le \cdots \le a_{d-1}$, there can be at most one value of $a_d$ such that $[a_1,\ldots,a_d]\in\cO(c,d)$, because any two $d$-dimensional grids that share the first $d-1$ dimensions are comparable in the dominance order $\preceq$.  Thus $\#_c(P)$ is bounded by the number of possible combinations of values of $a_1,\ldots,a_{d-1}$.  From the bound on each $a_k$ above, we therefore have
\begin{align*}
\#_c(P) & \le \left(\prod_{b=1}^{s-1} \; \prod_{k=\ell_{b-1}+1}^{\ell_b} O\left( c^{e(b)} \right)\right) \cdot \prod_{k=\ell_{s-1}+1}^{d-1} O\left(c^{e(s)}\right) \\
& = O\left(\prod_{b=1}^{s-1} \left(c^{e(b)}\right)^{\ell_b - \ell_{b-1}} \right)\cdot O\left(\left(c^{e(s)}\right)^{d-1-\ell_{s-1}}\right) \\
& = O\left( c^{h_1+h_2}\right)
\end{align*}
where $h_2 = e(s)(d-1-\ell_{s-1})$ and
\begin{align*}
h_1 & = \sum_{b=1}^{s-1} e(b) (\ell_b-\ell_{b-1}) \\
& = \sum_{b=1}^{s-1} 3^{\ell_{b-1}}\cdot 2^{\ell_b - \ell_{b-1} - 1} \cdot(\ell_b-\ell_{b-1}) \\
& \le \sum_{b=1}^{s-1} 3^{\ell_{b-1}}\cdot \frac{3^{\ell_b - \ell_{b-1}} - 1}{2} \\
& = \frac{1}{2} \sum_{b=1}^{s-1} \left(3^{\ell_b} - 3^{\ell_{b-1}}\right) \\
& = \frac{3^m - 1}{2},
\end{align*}
where $m = \ell_{s-1}$.  We also have
\begin{align*}
h_2 &= 3^{\ell_{s-1}}\cdot 2^{d - \ell_{s-1} - 1}\cdot (d - 1 - \ell_{s-1}) \\
&= 3^m \cdot 2^{d-m-1} \cdot (d-m-1),
\end{align*}
whence
$$
h_1 + h_2 = \frac{3^m - 1}{2} + 3^m \cdot 2^{d-m-1} \cdot (d-m-1).
$$
So our bound on the exponent of $c$ only depends on the value of $m$, which satisfies $0 \le m < d$.  It is more convenient to express $h_1+h_2$ in terms of $n := d-m$, where $n\in[d]$:
\begin{align*}
h_1+h_2 &= \frac{3^{d-n} - 1}{2} + 3^{d-n} \cdot 2^{n-1} \cdot (n-1) \\
&= \frac{3^d}{2} \cdot\frac{1 + 2^n(n-1)}{3^n} - \frac{1}{2}.
\end{align*}
It is easy to check that $(1+2^n(n-1))/3^n$ is greatest (and thus $h_1+h_2$ is greatest) when $n = 3$.  It follows that
\begin{align*}
h_1 + h_2 &\leq \frac{3^d}{2} \cdot \frac{1 + 2^3(3-1)}{3^3} - \frac{1}{2} \\
&= \frac{17\cdot 3^{d-3} - 1}{2},
\end{align*}
which proves the theorem.
\end{proof}

The first few values $(17\cdot 3^{d-3} - 1)/2$ are given in the Figure \ref{fig:exponenttable}.

\begin{figure}[htp]
\centering
\begin{tabular}{r|r}
$d$ & $(17\cdot 3^{d-3} - 1)/2$ \\ \hline
3 &  8 \\
4 & 25 \\
5 & 76 \\
6 & 229 \\
\end{tabular}
\caption{Table of upper bounds on $e$ so that $|\cO(c,d)| = O(c^e)$ for small $d$.} \label{fig:exponenttable}
\end{figure}

\section{Three Dimensions and Two Colors} \label{sec:threeD}

The following graph (Figure \ref{fig:surface}, generated using the Jmol module in SAGE) and table (Figure \ref{fig:a3table}) display upper bounds for the smallest $a_3$ so that $[a_1,a_2,a_3]$ is $2$-guaranteed.  All three graphical axes run from 3 to 130; the table includes only $3 \leq a_1 \leq 12$ and $3 \leq a_2 \leq 12$.  We believe these values to be very close to the truth; indeed, we have matching lower bounds in many cases, and lower bounds that differ from the upper bounds by at most 2 in many more cases.

\begin{figure}[htp]
\centering
\includegraphics[scale = .25]{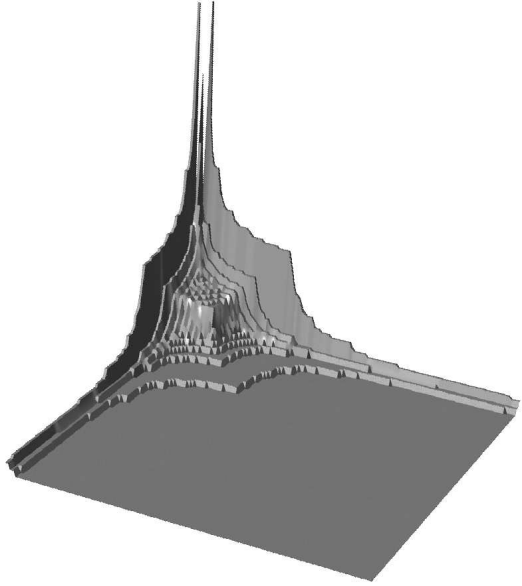}
\caption{Graph of upper bounds on $a_3$ so that $[a_1,a_2,a_3]$ is $2$-guaranteed.} \label{fig:surface}
\end{figure}

\begin{figure}[htp]
\centering
\begin{tabular}{l||l|l|l|l|l|l|l|l|l|l|}
  & 3 & 4 & 5 & 6 & 7  & 8  & 9 & 10 & 11 & 12 \\
\hline
\hline
3 &   &   &   &   & 127& 85 & 73& 68 & 67 & 67 \\
\hline
4 &   &   &   &   & 127& 85 & 73& 68 & 67 & 67\\
\hline
5 &   &   &101& 76& 53 & 47 & 46& 46 & 40 & 37\\
\hline
6 &   &   & 76& 76& 53 & 47 & 46& 46 & 40 & 37\\
\hline
7 &127&127& 53& 53& 53 & 46 & 40& 37 & 34 & 33\\
\hline
8 &85 & 85& 47& 47& 46 & 45 & 40& 37 & 34 & 33\\
\hline
9 &73 & 73& 46& 46& 40 & 40 & 37& 34 & 31 & 30\\
\hline
10&68 & 68& 46& 46& 37 & 37 & 34& 33 & 31 & 30\\
\hline
11&67 & 67& 40& 40& 34 & 34 & 31& 31 & 30 & 28\\
\hline
12&67 & 67& 37& 37& 33 & 33 & 30& 30 & 28 & 28\\
\hline
\end{tabular}
\caption{Table of bounds on $a_3$ so that $[a_1,a_2,a_3]$ is $2$-guaranteed.} \label{fig:a3table}
\end{figure}

A few different methods were applied to obtain these bounds.  First, the values $\Delta_j$, as in Section 3, were computed, and the least $a_3$ so that $\Delta_3 > 0$ was recorded.  In fact, this idea was improved slightly by applying the observation that, if some grid is $(2,t)$-guaranteed, then it is $(2,\ceil{t})$-guaranteed.  In some cases, this increases the value of $\Delta_j$.  Second, we used the simple observations that $c$-colorability is independent of the order of the $a_i$, and that $R \preceq R^\prime$ when $R$ is $c$-guaranteed implies that $R^\prime$ is $c$-guaranteed.  Third, we applied the following lemma.

\begin{lemma} If the grid $R = [a_1,\ldots,a_d]$ is $(c,t)$-guaranteed, then $R \times [\floor{cM/t}+1]$ is $c$-guaranteed, where $M = \prod_{j=1}^d \binom{a_j}{2}$
\end{lemma}
\begin{proof} Note that $K = \floor{cM/t}+1 > cM/t$ and is integral.  If we think of $R \times [K]$ as $K$ copies of $R$, then any $c$-coloring of $R \times [K]$ restricts to $K$ $c$-colorings of $R$.  Since $R$ is $(c,t)$-guaranteed, each of these $c$-colorings gives rise to $t$ monochromatic boxes. Hence, in $K$ colorings, there are at least $t (\floor{cM/t}+1) > cM$ monochromatic boxes.  Since there are only $M$ total boxes in each copy of $R$, and any monochromatic box can only be colored in $c$ different ways, there must be two identical boxes (in two different copies of $R$) which are monochromatic and have the same color.  This is precisely a monochromatic $(d+1)$-dimensional box in $R \times [K]$.
\end{proof}

Therefore, in order to obtain upper bounds on $[a_3]$ in the above table, we need to know the greatest $t$ for which $[a_1]\times[a_2]$ is $(2,t)$-guaranteed.  To that end, we define the following matrix:

\begin{definition} Let $M_r$ be the $2^r \times 2^r$ integer matrix whose rows and columns are indexed by all maps $f_j : [r] \rightarrow [2]$, $0 \leq j < 2^r$.  The $(i,j)$-entry of $M_r$ is defined to be
$$
\binom{|f_i^{-1}(1) \cap f_j^{-1}(1)|}{2} + \binom{|f_i^{-1}(2) \cap f_j^{-1}(2)|}{2}.
$$
Then define the quadratic form $Q_r : \R^{2^r} \rightarrow \R$ by $Q_r(\bv) = \bv^\ast M_r \bv$.  Let $\delta_r = (M_r(1,1),\ldots,M_r(2^r,2^r))$, the diagonal of $M_r$.
\end{definition}

\begin{proposition} Let $t$ be the least value of $Q_r(\bv)-\bv \cdot \delta_r$ over all nonnegative integer vectors $\bv \in \Z^{2^r}$ with $\bv \cdot \1 = s$.  Then $[r] \times [s]$ is $(c,t)$-guaranteed, and $t$ is the minimum value so that this is the case.
\end{proposition}
\begin{proof} Given a vector $\bv = (v_1,\ldots,v_r)$ satisfying the hypotheses, consider the $r \times s$ matrix $A$ with $v_j$ columns of type $f_j$ for each $j \in [r]$.  (We may identify $f_j$ with a column vector in $[2]^r$ in the natural way.)  It is easy to see that $Q_r(\bv)-\delta_r$ exactly counts twice the number of monochromatic rectangles in $A$, thought of as a $2$-coloring of the grid $[r] \times [s]$.
\end{proof}

We applied standard quadratic integer programming tools (XPress-MP) to minimize the appropriate programs.  Fortunately, for the cases considered, the matrix $M_r$ was positive semidefinite, meaning that the solver could use polynomial time convex programming techniques during the interior point search.  We conjecture that this is always the case.

\begin{conjecture} $M_r$ is positive semidefinite for $r \geq 3$.
\end{conjecture}

In particular, for $r=3$, the eigenvalues of $M_r$ are $0$, $1$, and $4$, with multiplicities $2$, $4$, and $2$, respectively.  For $4 \leq r \leq 9$, the eigenvalues are $0$, $2^{r-2}$, $2^{r-3}(r-2)$, $2^{r-2}(r-1)$, and $2^{r-4}(r^2-r+2)$, with multiplicities $2^r - r(r+1)/2$, $r(r-1)/2 - 1$, $r-1$, $1$, and $1$, respectively.  We conjecture that this description of the spectrum is valid for all $r \geq 4$.

\end{document}